\documentclass[twoside,a4paper,reqno,11pt]{amsart} 
\usepackage[top=28mm,right=30mm,bottom=28mm,left=30mm]{geometry}

\usepackage{amsfonts, amsmath, amssymb, mathrsfs, bm, latexsym, stmaryrd, array, mathtools, bold-extra}

\usepackage{etaremune}
\usepackage{enumitem}

\usepackage[]{hyperref}

\DeclareMathOperator{\ep}{\epsilon}

\DeclareMathOperator{\Gal}{Gal}
\DeclareMathOperator{\Aut}{Aut}

\DeclareMathOperator{\Irr}{Irr}

\DeclareMathOperator{\GL}{GL}
\DeclareMathOperator{\SL}{SL}

\newcommand{\FF}{\mathbb{F}}
\newcommand{\CCC}{{\mathcal C}}

\newcommand{\ZZ}{{\mathbb{Z}}}

\DeclareMathOperator{\Sp}{Sp}
\DeclareMathOperator{\SSS}{PSp}

\DeclareMathOperator{\PSL}{PSL}
\DeclareMathOperator{\LL}{PSL}

\DeclareMathOperator{\PSp}{PSp}

\newcommand{\cd}{{\mathrm {cd}}}
\newcommand{\cod}{{\mathrm {cod}}}

\numberwithin{equation}{section}

\newcommand{\Alt}{{\mathrm {A}}}

\newtheorem{theorem}{Theorem} 
\newtheorem{conj*}{Conjecture}

\newtheorem{conj}[theorem]{Conjecture}

\newtheorem{thm}{Theorem}[section] 
 
\newtheorem{lem}[thm]{Lemma}

\theoremstyle{definition}

\begin{document}

\title[Codegrees of finite simple groups]{A characterization of some finite simple  groups by their character codegrees}

\author{Hung P. Tong-Viet}
\address{H.P. Tong-Viet, Department of Mathematics and Statistics, Binghamton University, Binghamton, NY 13902-6000, USA}
\email{htongvie@binghamton.edu}
\dedicatory{}

\begin{abstract}  For a finite group $G$ and an irreducible complex character $\chi$ of $G$, the codegree of $\chi$ is defined by $\textrm{cod}(\chi)=|G:\textrm{ker}(\chi)|/\chi(1)$, where $\textrm{ker}(\chi)$ is the kernel of $\chi$. In this paper, we show that if $H$ is a finite simple exceptional group of Lie type or  a projective special linear group and $G$ is any finite group such that the character codegree sets of $G$ and $H$ coincide, then $G$ and $H$ are isomorphic.  
\end{abstract}

\thanks{}


\date{\today}
%\thanks{}


\maketitle

%-----------------------

\section{Introduction}

Let $G$ be a finite group and let $\chi$ be an irreducible complex character of $G$. Let $\ker(\chi)$ denote the kernel of $\chi$, that is, the set of all elements $g\in G$ such that $\chi(g)=\chi(1)$. The positive integer $\chi(1)$ is called the character degree  of $\chi$.
This is the degree of a complex representation of $G$ that affords the character $\chi$. In \cite{CH}, Chillag and Herzog call the number $|G|/\chi(1)$ a `character degree quotient', and obtain several interesting results concerning the structure of the finite groups $G$ by putting certain arithmetic  restriction on these character degree quotients. For example, they show that if $|G|/\chi(1)$ is a prime power for some irreducible character $\chi$ of $G$, then $G$ is not a finite nonabelian simple group (see \cite[Theorem 1]{CH}). In a subsequent paper \cite{CMM}, Chillag, Mann and Manz called the quotient $|G|/\chi(1)$ the `codegree of $\chi$'.
Among other results, they show that if $G$ has an irreducible character $\chi$ which satisfies  the same condition as above for a prime $p$, then $\chi$ is monomial and $F^*(G)=O_p(G)> 1$, where $F^*(G)$ is the generalized Fitting subgroup of $G$ and $O_p(G)$ is the largest normal $p$-subgroup of $G$.  This result was later reproved by Riese and Schmid in \cite{RS}.

Qian, Wang and Wei \cite{QWW} modified the definition of codegrees of characters by defining it to be $|G:\ker(\chi)|/\chi(1)$. Clearly, if $\chi$ is faithful, then the two definitions coincide. However, the new definition of codegrees which we will use in this paper is more compatible with taking quotients in the sense that the codegree of a quotient of a finite group $G$ is also the codegree of $G$. Now let $\cd(G)$ and $\cod(G)$ be the set of all degrees and codegrees of $G$, respectively. It is well known that these two sets have strong influence on the group structure.  In the late 1990s,  Huppert \cite{Huppert} conjectured that every finite nonabelian simple group is essentially determined 
by the set of its character degrees, that is, if $H$ is a finite nonabelian simple group and $G$ is a finite group such that $\cd(G)=\cd(H)$, then $G\cong H\times A$ for some abelian group $A$. This conjecture has been verified for all sporadic simple groups, the alternating groups of degree at least $5$ and low degree finite nonabelian simple groups of Lie type (see \cite{BTZ,HMTW,NgT} and the references therein).  In \cite{Tong24}, we are able to verify Huppert's conjecture for the remaining finite simple exceptional groups of Lie type.
 Motivated by Huppert's conjecture, Qian \cite[Problem 20.79]{KM} proposed the following:

\begin{conj}\label{conj:codegrees}
Let $H$ be a finite nonabelian simple group and let $G$ be a finite group such that $\cod(G)=\cod(H)$. Then $G\cong H.$
\end{conj}

This conjecture has been verified for all $26$ sporadic simple groups, the alternating groups, ${}^2{\rm B}_2(q^2),$ $ {}^2{\rm G}_2(q^2),$ $ {}^2{\rm F}_4(q^2), $ $\PSL_2(q),\PSL_3(q)$, $\textrm{PSU}_3(q),\PSp_4(q)$ ($q$ even) and other small finite nonabelian simple groups in \cite{Aha,BAK,DMSSY2,DMSSY1,Gintz,GZY,LY,Wang,Yang24}.
In a recent paper, Hung and Moret\'{o} \cite{HM} show that the set of character codegrees together with their multiplicity determines the finite nonabelian simple groups up to isomorphism. This is analogous to a result stating that  finite nonabelian quasisimple groups are uniquely determined by the structure of their complex group algebra or equivalently their character degrees together with their multiplicity obtained  in \cite{BNOT,TV12a,TV12b,TV12c}. The authors also show that if $S$ and $H$ are finite nonabelian simple groups such that $\cod(S)\subseteq \cod(H)$, then $S\cong H$.  This important result allows us to complete the verification of Conjecture \ref{conj:codegrees} for all finite simple exceptional groups of Lie type and ${\rm PSL}_n(q)$ for $n\ge 4$ and $q$  a prime power.

\begin{theorem}\label{th:main1}
Let $H$ be a finite nonabelian simple exceptional group of Lie type or ${\rm PSL}_n(q)$ with $n\ge 4$ and $q$ a prime power. Let $G$ be a finite group such that $\cod(G)=\cod(H).$ Then $G\cong H.$
\end{theorem}

We should mention that  this is the first time  a conjecture in this flavor has been confirmed for a family of finite simple groups of Lie type with unbounded Lie rank, defined over an arbitrary finite field.

We now describe our approach to the proof of Theorem \ref{th:main1}. Let $H$ be one of the finite nonabelian simple groups in the statement of the theorem defined over a finite field of size $q=p^f$ for some prime $p$. So $H$ is a finite simple group  of Lie type in characteristic $p$.  Let $G$ be a counterexample to Conjecture \ref{conj:codegrees} with  minimal order.  By \cite[Theorem 8.3]{HM}, $G$ has a unique minimal normal subgroup $N$ and $G/N\cong H$. By Lemma \ref{lem:minimal-normal}, $N$ is an elementary abelian $r$-group of order $r^m$ for some prime $r$ and $G$ is perfect.  Hence, either $N=Z(G)$ or $C_G(N)=N$. 
Since $G$ is perfect, the former case implies that $G$ is a finite quasisimple group with $G/Z(G)\cong H$ and $|Z(G)|=r$. This case can be eliminated by finding certain small degree faithful irreducible character $\chi$ of $G$ such that $\chi(1)/r$ is not a degree of $H$ (see Lemma \ref{lem:quasisimple}). 
In the latter case,  $N$ is a faithful irreducible $\FF_r H$-module and  $r^{2m}$ divides  $ |H|$ (see Lemma \ref{lem:2m}).   At this point, we can divide the proof into two cases according to whether $r$ is the defining characteristic  $p$ or not. For the cross characteristic cases, we can use the well-known Landazuri-Seitz-Zalesskii's bounds (see \cite[Table 5.3.A]{KL}) to obtain a lower bound for $m$ and check that $r^{2m}$ is too large comparing to the order of the Sylow $r$-subgroup of $H$ and we arrive at a contradiction.  For the defining characteristic cases,  we first get the lower bound $m\ge Df$, where $D$ is the dimension of the natural module or the minimal module for $H$ (see \cite{GT05}). For finite simple exceptional groups of Lie type, this bound is large enough to guarantee that  $r^{2m}\nmid |H|$  and we are done. However for ${\rm PSL}_n(q)$ this is not the case. To obtain a contradiction, we use some results on the representation theory of finite  groups of Lie type in defining characteristic to reduce it to the cases where $N$ is essentially the natural module for $\SL_n(q)$ and then using the fact that the second cohomology group $H^2(\SL_n(q),N)$ is trivial to deduce that $G\cong N:\SL_n(q)$ is a split extension and produce a faithful irreducible character of $G$ whose degree is not divisible by $p$,  which leads to a contradiction.

Our approach could be used to verify Conjecture \ref{conj:codegrees} for other  small rank classical groups. However, we might need a different method to verify the last step for large rank cases.

\medskip
\noindent
\textbf{Organization.} We collect some useful results  for the proof of the main theorem in Section \ref{sec2}.
We  verify Conjecture \ref{conj:codegrees} for the  finite simple exceptional groups of Lie type in Section \ref{sec4} and $\PSL_n(q)$ in Section \ref{sec5}. In the last section, we  consider the minimal counterexamples to Conjecture \ref{conj:codegrees} for the projective symplectic groups $\SSS_{2n}(q)$.

\medskip
\noindent
\textbf{Notation.} Our notation is standard. We follow \cite{KL} for  the notation of simple groups and \cite{HB,Isaacs} for  the representation and character theories of finite groups. For a positive integer $n$, we write $\ZZ_n$ for a cyclic group of order $n.$ For a prime $p$ and a positive integer $n$, the $p$-part of $n$, denoted by $n_p$ is the largest power of $p$ dividing $n$ and $n_{p'}=n/n_p$ is the $p'$-part of $n.$ For a finite group $G$,  let $M(G)$ denote the Schur multiplier of $G$.  We denote by $\Irr(G)$ the set of complex irreducible characters of $G$. For $\chi\in\Irr(G)$,   the codegree of $\chi$ is defined by $\cod(\chi)=|G:\ker(\chi)|/\chi(1)$. If $N$ is a normal subgroup of a finite group $G$ and $\theta\in\Irr(N)$, we write $\Irr(G|\theta)$ for the set of all irreducible characters of $G$ lying above $\theta.$
As usual, $Z(G)$ is the center of $G$ and $C_G(A)$ is the centralizer of $A$ in $G$, where $A$ is a subgroup of $G$. Finally, a finite group $G$ is 
quasisimple if $G$ is perfect and $G/Z(G)$ is a finite nonabelian simple group.

\section{Preliminaries}\label{sec2}

We prove several lemmas that are needed for the proof of the main theorem. We first deduce some properties of a minimal counterexample to Conjecture \ref{conj:codegrees}.

\begin{lem}\label{lem:minimal-normal}
Let $H$ be a finite nonabelian simple group and let $G$ be a finite group with $\cod(G)=\cod(H)$. Suppose that $G$ has a unique minimal normal subgroup $N$ such that $G/N\cong H.$ Then $G$ is perfect and $N$ is an elementary abelian $r$-group, for some prime $r$.
\end{lem}

\begin{proof}
We have $N\cong S^k$ for some finite simple group $S$ and integer $k\ge 1.$ Assume that $S$ is nonabelian.
It is well-known that $S$ has a nontrivial irreducible character $\theta$ that extends to $\Aut(S)$ and thus the irreducible character $\theta^k\in\Irr(N)$ extends to $\chi\in \Irr(G)$ (see \cite{TV12a}). Now $\chi$ is faithful by the uniqueness of $N$. Hence $|G|/\chi(1)=|H|/\varphi(1)$ for some $\varphi\in\Irr(H)$. 
It follows that $\varphi(1)=\theta(1)^k/|S|^k$. However, since $S$ is a finite nonabelian simple group, $\theta(1)<|S|$ and so $\varphi(1)=(\theta(1)/|S|)^k<1$, which is a contradiction. 
Therefore, $S$ is abelian and hence $N$ is an elementary abelian $r$-group for some prime $r$. The fact that $G$ is perfect follows from \cite[Theorem 2.3]{HM}.
\end{proof}

Assume the hypothesis and the conclusion of Lemma \ref{lem:minimal-normal}. Since $N$ is the unique minimal normal subgroup of $G$ and is elementary abelian, $N\leq C_G(N)\unlhd G$. Since $G/N\cong H$ is a finite nonabelian simple group, either $C_G(N)=G$ and thus $G$ is a finite quasisimple group with center $N$ or $C_G(N)=N$. To handle the former case, we use the following observation.

\begin{lem}\label{lem:quasisimple}
Let $H$ be a finite nonabelian simple group and let $G$ be a finite quasisimple group with $N=Z(G)$ a cyclic group of prime order $r$ such that $G/N\cong H$. Assume $\cod(G)=\cod(H)$. Then for every faithful  character $\chi\in\Irr(G)$, $\chi(1)/r\in\cd(H)$. 
\end{lem}

\begin{proof}
Let $\chi\in\Irr(G)$ be a faithful  character of $G$. Since $\cod(\chi)=|G|/\chi(1)\in\cod(H)$, $|G|/\chi(1)=|H|/\varphi(1)$ for some $\varphi\in\Irr(H)$. Hence $\varphi(1)=\chi(1)|H|/|G|=\chi(1)/r\in\cd(H)$. The proof is now complete.
\end{proof}

Therefore, to show that $G$ is not quasisimple, for each prime divisor $s$ of  $|M(H)|$, where $M(H)$ denotes the  Schur multiplier of $H$, and every cyclic central extension $L$ of $H$ with $|Z(L)|=s$, we need to find a faithful irreducible character $\chi$ of $L$ such that
$\chi(1)/s\not\in\cd(H)$. 
The next lemma is key to our approach to verify Conjecture \ref{conj:codegrees}. 

\begin{lem}\label{lem:2m}
Let $H$ be a finite nonabelian simple group and let $G$ be a finite group with $\cod(G)=\cod(H)$. Suppose that $G$ has a unique minimal normal subgroup $N$, where $N$  is an elementary abelian $r$-group of order $r^m$ for some prime $r$,  $G/N\cong H$ and $C_G(N)=N$. Then $N$ is a nontrivial  faithful irreducible $\FF_rH$-module and  the following hold.
\begin{enumerate}
\item[$(i)$]  $r^{2m}$ divides $|H|$.
\item[$(ii)$] If $\chi\in\Irr(G)$ is faithful, then $r^m\mid \chi(1)$ and $\chi(1)/r^m\in\cd(H)$. 
\end{enumerate}
\end{lem}

\begin{proof}
Since $N$ is a minimal normal elementary abelian $r$-subgroup of $G$ and $C_G(N)=N$,  $H\cong G/N=G/C_G(N)$ is isomorphic to an irreducible subgroup of $\GL(N)\cong \GL_m(\FF_r)$. Therefore, $N$ is a nontrivial faithful irreducible $\FF_rH$-module. 

 We  show that $r^{2m}\mid |H|$. We know that $G$ acts on $\Irr(N)\cong N$ by conjugation and let $\{\mathcal{O}_i\}_{i=0}^t$ be $G$-orbits of this action with $\mathcal{O}_0=\{1_N\}.$ Since $G$ does not centralize any nontrivial element of $N$, $G$ does not stabilize any nontrivial irreducible character of $N$ and thus $\mathcal{O}_0$ is the only $G$-orbit on $\Irr(N)$ that has size $1$. For any  $j\ge 1$, fix $\theta_j\in \mathcal{O}_j$ and let $T_j=I_G(\theta_j)$. Then $|\mathcal{O}_j|=|G:T_j|$. If $r\mid |\mathcal{O}_j|$ for all  $j$ with $1\leq j\leq t$, then $r\mid |N|-1= r^m-1$, which is impossible. Therefore, there exists an index $j_0$ with $1\leq j_0\leq t$ such that $r\nmid |\mathcal{O}_{j_0}|$. Let $\theta=\theta_{j_0}$ and $T=T_{j_0}$.

Write $\theta^T=\sum_{i=1}^ae_i\phi_i$, where $\Irr(T |\theta)=\{\phi_1,\phi_2,\dots,\phi_a\}$ and $e_i=\phi_i(1)$ for all $i\ge 1$. By Clifford's correspondence theorem, $\theta^G=\sum_{i=1}^ae_i\chi_i$, where $\chi_i:=\phi_i^G\in\Irr(G|\theta)$. For each $1\leq k\leq a,$  $\chi_k$ is faithful  and thus $\cod(\chi_k)=|G|/\chi_k(1)=|H|/\varphi_k(1)$ for some $\varphi_k\in\Irr(H).$ It follows that $\chi_k(1)=r^m\varphi_k(1)$, whence $|G:T|\phi_k(1)=r^m\varphi_k(1)$.  Since $r$ does not divide $ |\mathcal{O}_{j_0}|=|G:T|$, $r^m\mid \phi_k(1)$ and hence $\phi_k(1)=r^mc_k$ for some positive integer $c_k.$ Now
 $$|T:N|=\theta^T(1)=\sum_{i=1}^ae_i\phi_i(1)=\sum_{i=1}^a\phi_i(1)^2=r^{2m}\sum_{i=1}^ac_i^2.$$ Therefore, $r^{2m}$ divides $|H|$ as wanted. 
 
 Finally, let $\chi\in\Irr(G)$ be a faithful character of $G$. Then $\cod(\chi)=|G|/\chi(1)=|H|/\varphi(1)$ for some $\varphi\in\Irr(H)$. It follows that $\chi(1)=|N|\varphi(1)=r^m\varphi(1)$ and (ii) follows.
\end{proof}

Let $q$ be a prime power and let $n\ge 2$ be an integer. A prime divisor $r$ of $q^n-1$ is called a primitive prime divisor of $q^n-1$ if $r$ divides $q^n-1$ but $r$ does not divide $q^i-1$ for any integer $i$ with $1\leq i<n$. By Zsigmondy's theorem \cite{Zsig}, such a prime $r$ exists unless $(n,q)=(6,2)$ or $n=2$  and $q$ is a Mersenne prime.  For a prime $r\neq p$, we write $d_q(r)$ for the order of $q$ modulo $r$. In this case, $r$ is a primitive prime divisor of $q^j-1$ for $j=d_q(r)$ and hence $j\mid (r-1)$; in particular, $r\ge j+1.$

We also need the following number theoretic results that will be used to compute the orders of the Sylow subgroups of finite simple groups of Lie type.

\begin{lem}\label{lem:Sylow}
Let $q$ be a prime power, let $d$ be a positive integer and let $r$ be a prime divisor of $q-\epsilon$, where $\epsilon=\pm$.
\begin{itemize}\addtolength{\itemsep}{0.2\baselineskip}
\item[{\rm (i)}] If $r=2$, then
\[
(q^d-\epsilon)_r = \left\{\begin{array}{ll}
(q-\epsilon)_r & \mbox{if $d$ is odd} \\
(q^2-1)_r(d/2)_r & \mbox{if $d$ is even and $\epsilon=+$} \\
2 & \mbox{if $d$ is even and $\epsilon=-$.}
\end{array}\right.
\]
\item[{\rm (ii)}] If $r$ is odd, then
\begin{align*}
(q^d-\epsilon)_r & = \left\{ \begin{array}{ll}
1 & \mbox{if $d$ is even and $\epsilon=-$} \\
(q-\epsilon)_r(d)_r & \mbox{otherwise}
\end{array}\right. \\
(q^d+\epsilon)_r & = \left\{ \begin{array}{ll}
(q-\epsilon)_r(d)_r & \mbox{if $d$ is even and $\epsilon=-$} \\
1 & \mbox{otherwise.}
\end{array}\right.
\end{align*}
\end{itemize}
\end{lem}

\begin{proof}
This is Lemma 2.1 in \cite{BBGT}.
\end{proof}
The orders of the Sylow $r$-subgroups of finite simple classical groups in cross characteristic can be found in \cite{Weir} and \cite{Stather}.

If $r\ge 2$ and $x\ge 1$ are integers, we will use the inequality $x<r^x$ or equivalently $\log_r(x)<x$ regularly. If $1\leq a_1\leq a_2\leq \dots\leq a_k$ are integers with $k\ge 1$ and $q\ge 2$, then $$  \prod_{i=1}^k(q^{a_i}-1)< q^{a_1+\dots a_k}.$$ Moreover, if $a\ge 1$ is an integer, then $q^a\pm 1\ge \frac{1}{2}q^a$.
More generally, if $2\leq a_1<\dots<a_k$ are integers and $\ep_1,\ep_2,\dots,\ep_k\in\{\pm 1\}$, then \[\frac{1}{2}<\frac{(q^{a_1}+\ep_1)\dots(q^{a_k}+\ep_k)}{q^{a_1+\dots+a_k}}<2.\] For a proof, see Lemma 2.1 in \cite{TZ1}.

%-------------Simple exceptional groups of Lie type---
\section{Finite simple exceptional groups of Lie type}\label{sec4}

Let $q=p^f$, where $p$ is a prime and $f\ge 1$ is an integer.  We use the convention   ${\rm E}_6^+(q)$ for ${\rm E}_6(q)$ and ${\rm E}_6^-(q)$ for ${}^2{\rm E}_6(q)$. In this section, we verify Conjecture \ref{conj:codegrees} for the following simple exceptional groups of Lie type:
\begin{equation}\label{eqn1}
{\rm G}_2(q), {\rm F}_4(q), {\rm E}^\epsilon_6(q),{\rm E}_7(q), {\rm E}_8(q),{}^3{\rm D}_4(q),{}^2{\rm B}_2(q),{}^2{\rm F}_4(q),{}^2{\rm G}_2(q)
\end{equation}

We exclude   the groups $ {}^2{\rm B}_2(2),{\rm G}_2(2),{}^2{\rm G}_2(3) $ and ${}^2{\rm F}_4(2)$ from our list above as they are not simple. 
The order formulas for simple groups in \eqref{eqn1} can be found in \cite[Table 5.1.B]{KL} and the Schur multipliers of these groups can be found in \cite[Theorem 5.1.4]{KL}. The simple groups with exceptional Schur multipliers are listed in \cite[Table 5.1.D]{KL}. The lower bounds for the degrees of nontrivial absolutely irreducible projective representations of finite simple groups of Lie type in cross characteristic can be found in \cite[Table 5.3.A]{KL} and finally the minimal modules for these simple groups can be found in \cite[Table 5.4.C]{KL}.  

If $H$ is one of the simple groups: ${\rm E}_6(q), {}^2{\rm E}_6(q)$ ($q>2$) or ${\rm E}_7(q)$, then $M(H)$ is of order $\gcd(3,q-1),\gcd(3,q+1)$ and $\gcd(2,q-1),$ respectively. The Schur multiplier of $ {}^2{\rm E}_6(2)$ is $\ZZ_2\times \ZZ_6$. To handle these cases, we  need the data available on L\"{u}beck's website \cite{Lubweb}.
For each positive integer $k$, denote by $\Phi_m=\Phi_m(q)$ the cyclotomic polynomial of degree $m$ evaluated at $q$.

\begin{lem}\label{lem:Schur-cover}
Let $H$ be one of the simple groups ${\rm E}_6(q),{}^2{\rm E}_6(q)$ and ${\rm E}_7(q)$, where $q=p^f$ for some prime $p$. If $G$ is a finite quasisimple group such that $|Z(G)|=r$ is a prime and $G/Z(G)\cong H$, then $\cod(G)\neq \cod(H).$
 \end{lem}

\begin{proof}

We first consider the case when $H$ has an exceptional Schur multiplier. So, let $H={}^2{\rm E}_6(2)$. Now $M(H)\cong \ZZ_2\times \ZZ_6$. We only need to consider the cases  $G\cong 3.{}^2{\rm E}_6(2)$ and $G\cong 2.{}^2{\rm E}_6(2)$.
In the first case, $3.{}^2{\rm E}_6(2)$ has a faithful irreducible character of degree $46683$ and $46683/3=15561$ is not a degree of $H$. The smallest nontrivial degree of ${}^2{\rm E}_6(2)$ is $1938.$ Similarly,  $2.{}^2{\rm E}_6(2)$ has a faithful irreducible character of degree $2432$ and $2432/2=1216$ is not a degree of $H$ as well. Thus by Lemma \ref{lem:quasisimple}, $\cod(G)\neq \cod(H)$.

Now assume that $H$ is one of the simple groups in the lemma with $q>2$ if $H\cong {}^2{\rm E}_6(q)$. Assume also that $H$ has a nontrivial Schur cover. Thus $|M(H)|$ is a prime which is 3 in the first two cases and $2$ in the last case. Let $L$ be the full covering group of $H$ and let $L^*$ be the dual of $L$, so $L$ is a simply connected group and $L^*$ is the adjoint group.
We can also assume that $H=[L^*,L^*]$ and $L/Z(L)\cong H.$ Note that $Z(L)\cong \ZZ_r$ and $|L^*:H|=r$.

(a) Assume that $H\cong {\rm E}_7(q)$. In this case, $d=\gcd(2,q-1)>1$, so $q\ge 3$ is odd. Note that $L={\rm E}_7(q)_{sc}$ and $L^*={\rm E}_7(q)_{ad}.$ By using \cite{Lubweb}, $L$ has a degree $$D=\Phi_1^4\Phi_3^2\Phi_4\Phi_5\Phi_6\Phi_7\Phi_8\Phi_9\Phi_{12}\Phi_{18}$$ with multiplicity $(q-1)/2$.
Now $D/2$ is not a degree of $L$, so it is not a degree of $H$. Note that $\cd(H)\subseteq \cd(L).$  We are done if we can show that $L$ has a faithful irreducible character of degree $D$. Assume by contradiction that every irreducible character of $L$ of degree $D$ is unfaithful, so they are characters of $H$. Now $L^*$ also has irreducible characters of degree $D$ with multiplicity $(q-1)/2$ if $3\nmid (q-1)$ and $(q-3)/2$ if $3\mid (q-1).$ Furthermore, $L^*$ has no irreducible character of degree $2D.$ Therefore, every irreducible character of $H$ of degree $D$ extends to $L^*$ but then $L^*$ would have at least $2(q-1)/2=q-1$ irreducible character of degrees $D$, which is not the case. Hence $L$ must have some faithful irreducible character of degree $D$ as required. 

The remaining two cases can be argued similarly using the degree $D$ of $L$ with the given multiplicity.

(b) Assume  $H\cong {\rm E}_6(q)$. Then $3\mid (q-1)$ and so $q\equiv 1,4$ (mod 6).  By \cite{Lubweb}, $L$ has a degree $$D=\Phi_3^2\Phi_6\Phi_9\Phi_{12}$$ with multiplicity $q-2$ and $L$ has no character of degree $D/3$. Also the multiplicity of $D$ in  $L^*$ is $q-4$ and $L^*$ has no character of degree $3D$. 

(c) Assume  $H\cong {^2}{\rm E}_6(q)$ with $q>2$.  Then $3\mid (q+1)$ and so $q\equiv 2,5$ (mod 6).  By \cite{Lubweb}, $L$ has a degree $$D=\Phi_3\Phi_6^2\Phi_{12}\Phi_{18}$$ with multiplicity $q$ and $L$ has no character of degree $D/3$. Moreover, the multiplicity of $D$ in  $L^*$ is $q-2$ and $L^*$ has no character of degree $3D$. 
\end{proof}

We now prove the main theorem in this section.

%---Main theorem----------
\begin{thm}\label{th:exceptional}
Let $H$ be one of the simple groups in \eqref{eqn1}  and let $G$ be a finite group such that $\cod(G)=\cod(H)$. Then $G\cong H.$
\end{thm}

\begin{proof}
Let $G$ be a counterexample to the theorem with minimal order.  Then $\cod(G)=\cod(H)$ but $G\not\cong H$.   By Lemma \ref{lem:minimal-normal}, $G$ is perfect and $N$ is the unique minimal normal subgroup of $G$ which is an elementary abelian $r$-group for some prime $r$. Let $|N|=r^m$. Then either $G$ is quasisimple with $N=Z(G)$ or $C_G(N)=N$.  In the former case, $G$ is a finite quasisimple group with $G/Z(G)\cong H$, where $|Z(G)|=r$. Note that if $M(H)=1$, then this case cannot occur. In the latter case, by Lemma \ref{lem:2m}, $N$ is a faithful irreducible $\FF_r H$-module and $r^{2m}\mid |H|$. Hence $2m\leq \log_r(|H|_r)$. We now consider each possibility.

\smallskip\noindent
\textbf{Case $H\cong {\rm E}^{\epsilon}_6(q),q=p^f$} By Lemma \ref{lem:Schur-cover}, $G$ is not a quasisimple group. Hence $C_G(N)=N$ and $N$ is a faithful irreducible $\FF_r H$-module with $2m\leq \log_r(|H|_r)$. Let $d=\gcd(3,q-\epsilon).$ By \cite[Table 5.1.B]{KL}, we have  \[|H|=\frac{1}{d}q^{36}(q^2-1)(q^5-\epsilon)(q^{6}-1)(q^{8}-1)(q^{9}-\epsilon)(q^{12}-1).\] We now consider the cases $r=p$ and $r\neq p$ separately.

(1a) Case   $r\neq p$.  By \cite[Theorem 5.3.9]{KL}, $m\ge q^{9}(q^2-1).$ Let $j=d_q(r)$. Then $$j\in\{1,2,3,4,5,6,8,9,12,10,18\}.$$

(i) Assume   $r=2$. Then $q$ is odd,  and  $j=1$. Here $|H|_2\leq 2^4(q^2-1)_2^6$ and  so $$\log_2(|H|_2)=4+6\log_2((q^2-1)_2)\leq 6q^2.$$ Now $$2m\ge 2q^{9}(q^2-1)>6q^2\ge \log_2(|H|_2)$$ for any $q\ge 3$, a contradiction.

\smallskip
(ii) Assume  $r\ge 3$ is odd.
Assume $1\leq j\leq 6$. Then $$|H|_r\leq   (q^j-1)_r^{6}(6\cdot 10\cdot 12 \cdot 18)_r\leq r^4(q^j-1)_r^6$$ and  so $\log_r(|H|_r)\leq  4+6\log_r((q^j-1)_r)\leq 4+6(q^j-1)\leq 6q^6.$ Now $$2m\ge 2q^{9}(q^2-1)>6q^6\ge \log_r(|H|_r),$$ a contradiction. 
Assume that $j=9$. Then $r\ge 11$ and $\epsilon=+$. So  $|H|_r=  (q^9-1)_r$ and  hence $\log_r(|H|_r)\leq q^9-1\leq q^9.$  Now $2m\ge 2q^{9}(q^2-1)>q^9\ge \log_r(|H|_r)$. 
For the remaining values of  $j$, $8\leq j\leq 18$ and $j$ is even. So  $r\ge 11$ and $|H|_r=(q^j-1)_r=(q^{j/2}+1)_r$ and  so $\log_r(|H|_r)\leq q^{j/2}+1\leq q^9+1\leq 2q^9.$ Hence $2m\ge 2q^{9}(q^2-1)>2q^9\ge \log_r(|H|_r)$, a contradiction.

(1b) Case $r=p.$ 
 By \cite[Table 5.4.C]{KL}, the smallest dimension of nontrivial irreducible projective representations of ${\rm E}^\epsilon_6(q)$ over $\mathbf{k}=\overline{\FF}_p$ is $27$. By \cite[Lemma 4.2]{GT05}, $|N|=p^m\ge p^{36f}$, whence $2m\ge 54f$. Now $|H|_p=p^{36f}$ and thus since $2m\ge 54f>36f$, whence $p^{2m}\nmid |H|$.

\smallskip\noindent
\textbf{Case $H\cong {\rm E}_7(q),q=p^f$}.  The argument for this case is similar, we skip the details.

\smallskip\noindent
\textbf{Case $H\cong {\rm F}_4(q),q=p^f$}.

\smallskip

(1) Assume $G$ is  quasisimple  with $N=Z(G)$.
If $q>2$, then  the Schur multiplier of $\mathrm{F}_4(q)$ is trivial (see \cite[Theorem 5.1.4]{KL}). For $q=2,$ the universal cover of ${\rm F}_4(2)$ is $2.{\rm F}_4(2)$ and so $G\cong 2.{\rm F}_4(2)$. By GAP \cite{GAP}, $2.{\rm F}_2(2)$ has a faithful irreducible character $\chi$ of degree $52$ and the smallest nontrivial character degree of $H$ is $833$. Thus $52/2=26\not\in \cod(H)$, contradicting Lemma \ref{lem:quasisimple}. 

(2) Assume $C_G(N)=N$.
By Lemma \ref{lem:2m}, $N$ is a faithful irreducible $\FF_r H$-module and $r^{2m}\mid |H|$. It follows that $2m\leq \log_r(|H|_r).$ We have     \[|H|=q^{24}(q^2-1)(q^6-1)(q^8-1)(q^{12}-1).\]

\smallskip

(2a) Case  $r\neq p$. 
For $q=2$, we have $m\ge 52$ by \cite{HissMalle}. Since $|{\rm F}_2(2)|=2^{24}\cdot 3^6\cdot 5^2\cdot 7^2\cdot 13\cdot 17,$  $r^{2m}$ cannot divide $|H|$. So we assume that $q>2$. By \cite[Theorem 5.3.9]{KL}, we have $m\ge q^6(q^2-1)$ if $q$ is odd and $m\ge q^7(q^3-1)(q-1)/2$ if $q>2$ is even. Note that $q^7(q^3-1)(q-1)/2>q^6(q^2-1)$. Let $j=d_q(r)$. Then  $j$ divides $8$ or $12$.

(i) Assume  $r=2$. Then $q$ is odd and $2\mid (q-1)$, so $j=1$. Here $|H|_2=2^3(q^2-1)_2^4$. So $$\log_2(|H|_2)=3+4\log_2((q^2-1)_2)<4+4(q^2-1)=4q^2.$$ Now $2m\ge 2q^6(q^2-1)>4q^2 \ge \log_2(|H|_2)$, a contradiction.

\smallskip
(ii) Assume  $r\ge 3$ is odd.

Case $j=1$. Then $|H|_r= (q-1)_r^{4}(6\cdot 12)_r\leq r^2(q-1)_r^4$ and  so $$\log_r(|H|_r)\leq  2+4\log_r((q-1)_r)\leq 2+4(q-1).$$ Clearly $2m\ge 2q^6(q^2-1)>2+4(q-1)>\log_r(|H|_r)$, a contradiction. 

Case $j=2$. Then $|H|_r= (q^2-1)_r^{4}(6\cdot 12)_r\leq r^2(q^2-1)_r^4$ and  so $$\log_r(|H|_r)\leq  2+4\log_r((q^2-1)_r)\leq 2+4(q^2-1).$$ Clearly $2m\ge 2q^6(q^2-1)>2+4(q^2-1)>\log_r(|H|_r)$, a contradiction

For the remaining cases, since $j\ge 3$, $r\ge 5$, $r$ divides neither $8$ nor $12$.  If $j\leq 6$, then $|H|_r\leq (q^j-1)_r^3$; therefore, $\log_r(|H|_r)\leq 3\log_r(q^j-1)\leq 3(q^6-1)<3q^6$ and hence $2m\ge 2q^6(q^2-1)>3q^6\ge \log_r(|H|_r)$.
Finally, if $7\leq j\leq 12$, then $j=8$ or $12$ and hence $$\log_r(|H|_r)\leq \log_r((q^j-1)_r)\leq \log_r((q^{j/2}+1)_r)\leq q^{j/2}+1\leq q^6+1.$$ Now $2m\ge q^6(q^2-1)>q^6+1>\log_r(|H|_r)$, which is a contradiction.

\smallskip

(2b) Case $r=p.$  By \cite[Table 5.4.C]{KL}, the smallest dimension of nontrivial irreducible representations of ${\rm F}_4(q)$ over $\overline{\FF}_p$ is $26-\delta_{p,3}\ge 25$. By \cite[Lemma 4.2]{GT05}, $|N|=p^m\ge p^{25f}$, whence $m\ge 25f$. Now $|H|_p=p^{24f}$ and thus $2m\ge 50f>24f=\log_p(|H|_p)$, a contradiction.

\smallskip\noindent
\textbf{Case $H\cong  {\rm E}_8(q),  {\rm G}_2(q) (q>2), {}^3{\rm D}_4(q),q=p^f$}. The arguments for these cases are similar to the previous one. We  omit the details.

\smallskip\noindent
\textbf{Case $H\cong {}^2{\rm B}_2(q),{}^2{\rm F}_4(q),{}^2{\rm G}_2(q)$}. These cases have been verified elsewhere.
For completeness, we include the proofs here.

By \cite[Theorem 5.1.4]{KL}, the Schur multiplier of $H$ is trivial except for ${}^2{\rm B}_2(8)$, where its Schur multiplier is a Klein four-group $\ZZ_2\times \ZZ_2$. By GAP \cite{GAP}, the quasisimple group $2.{}^2{\rm B}_2(8)$ has a faithful irreducible character of degree $40$ but $40/2=20$ is not a degree of ${}^2{\rm B}_2(8)$. By Lemma \ref{lem:quasisimple}, $G$ cannot be a $2.{}^2{\rm B}_2(8)$. Thus $C_G(N)=N$ and so by Lemma \ref{lem:2m}, we have $2m\leq \log_r(|H|_r)$.

\smallskip
(i) Assume $H\cong {}^2{\rm B}_2(q)$ with $q=2^{2n+1},n\ge 1$. We have $$| {}^2{\rm B}_2(q)|=q^2(q-1)(q^2+1)=q^2(q-1)(q-2^{n+1}+1)(q+2^{n+1}+1).$$ 
Note that the last three factors are pairwise coprime.

Assume $r=2$. By \cite[Table 5.4.C]{KL}, the minimal module for $H$ is of dimension $4$ and so by \cite[Lemma 4.2]{GT05},  $|N|\ge q^4$ or $m\ge 4(2n+1)$. Hence $2m\ge 8(2n+1)>2(2n+1)=\log_2(|H|_2)$. 

Assume that $r\ge 3$ is odd. Assume that $n=1$. Then $H={\rm B}_2(8)$ and so by \cite[Table 5.3.A]{KL}, $m\ge 8.$ Now $|H|=2^6\cdot 5\cdot 7\cdot 13$ and thus $r^{2m}$ cannot divide $|H|$ for any odd prime $r$. Thus we can assume $n\ge 2$ and so $m\ge 2^n(q-1)$ by \cite[Table 5.3.A]{KL}.  For each odd prime $r$ dividing $|H|$, $r$ divides exactly one of the last three factors in   the order formula for $H$ above. Hence $|H|_r\leq q+2^{n+1}+1.$ However, we can check that $2m\ge 2^{n+1}(q-1)>q+2^{n+1}+1\ge \log_r(|H|_r)$.

\smallskip
(ii) Assume $H\cong {}^2{\rm G}_2(q)$ with $q=3^{2n+1},n\ge 1$. Note that \[|{}^2{\rm G}_2(q)|=q^3(q-1)(q^3+1)=q^3(q-1)(q+1)(q^2-q+1).\]

Assume that $r=3$. As in the previous case, note that the minimal module for $H$ is of dimension $7$, so $m\ge 7(2n+1)$. Hence $2m\ge 14(2n+1)>3(2n+1)=\log_3(|H|_3)$.

Assume that $r\neq 3.$ By \cite[Table 5.3.A]{KL}, $m\ge q(q-1)$. Assume $r=2$. Then $|H|_2=(q^2-1)_2=8$ since $q^2\equiv 9$ (mod 16). Since $q\ge 27$, it is clear that $2m\ge 2q(q-1)>3$ and thus $2^{2m}$ does not divide $|H|$. Now assume $r>2$. Then $r\ge 5$ and thus $r$ divides exactly one of the factors $q-1,q+1$ and $q^2-q+1$. Hence $|H|_r\leq q^2-q+1$. Now $2q(q-1)-(q^2-q+1)=q^2-q-1=q(q-1)-1>0$ since $q\ge 27$. Therefore, $2m\ge 2q(q-1)>\log_r(|H|_r)$.

\smallskip
(iii) Assume $H\cong {}^2{\rm F}_2(q)$ with $q=2^{2n+1},n\ge 1$. So $q\ge 8.$
We have  \[|{}^2{\rm F}_4(q)|=q^{12}(q-1)(q^3+1)(q^4-1)(q^6+1).\]

Assume $r=2$. As in the previous case,  the minimal module for $H$ is of dimension $26$, so $m\ge 26(2n+1)$. Hence $2m\ge 52(2n+1)>12(2n+1)=\log_2(|H|_2)$, a contradiction.

Assume that $r\ge 3.$ By \cite[Table 5.3.A]{KL}, $m\ge 2^nq^4(q-1)$. Let $j=d_q(r)$. Then $j\in\{1,2,4,6,12\}.$ 

Case $j=1,2,4$. Then $|H|_r\leq   (q^j-1)_r^{2}(3\cdot 6)_r\leq r^2(q^j-1)_r^2$ and  so $$\log_r(|H|_r)\leq  2+2\log_r((q^j-1)_r)\leq 2+2(q^2+1)\leq 4q^2.$$ Now $2m\ge 2^{n+1}q^{4}(q-1)>4q^2\ge \log_r(|H|_r)$. 
Assume $j=6$. Then $r\ge 7$ and $|H|_r\leq   (q^3+1)_r= (q^2-q+1)_r$ and  so $\log_r(|H|_r)\leq   q^2-q+1\leq 2q^2.$ Now $2m\ge 2^{n+1}q^{4}(q-1)>2q^2\ge \log_r(|H|_r)$. 
Assume $j=12$. Then $r\ge 13$ and $|H|_r\leq   (q^6+1)_r= (q^4-q^2+1)_r$ and  so $\log_r(|H|_r)\leq   q^4-q^2+1\leq 2q^4.$ Now $2m\ge 2^{n+1}q^{4}(q-1)>2q^4\ge \log_r(|H|_r)$. 
The proof is now complete.
\end{proof}

\section{Projective special linear groups}\label{sec5}

We assume $H\cong \PSL_n(q)$ with $n\ge 4$ and $q=p^f$, where $p$ is a prime and $f\ge 1$. Moreover, we assume $H$ is not isomorphic to $\PSL_4(2)$ as $\PSL_4(2)\cong \Alt_8$. By \cite[Theorem 5.1.4]{KL}, the Schur multiplier $M(H)$ of $H$ is cyclic of order $d=\gcd(n,q-1)$. So $L=\SL_n(q)$ is the universal cover of $H$.

\begin{thm}\label{th:Ln}
Let $H=\PSL_n(q)$, where $n\ge 4,f\ge 1$ and $(n,q)\neq (4,2).$ Let $G$ be a finite group such that $\cod(G)=\cod(H)$. Then $G\cong H.$
\end{thm}

\begin{proof}
Let $G$ be a counterexample to the theorem with minimal order.  Then $\cod(G)=\cod(H)$ but $G\not\cong H$.  By  Lemma \ref{lem:minimal-normal}, $G$ is perfect,  $G$ has a normal elementary abelian $r$-subgroup $N$  of order $r^m$ and $G/N\cong H$. Hence either $G$ is a quasisimple group with center $N$ of order $r$ or $C_G(N)=N$.

\smallskip

(1) Assume that $G$ is quasisimple with center $N\cong \ZZ_r.$  Recall that the center of $\SL_n(q)$ is cyclic of order $d=\gcd(n,q-1)$ and that $Z(\SL_n(q))$ has a unique subgroup of order $e$ for each divisor $e\ge 1$ of $d$. Since $G$ is a central quotient of $\SL_n(q)$, we deduce that $r\mid d$ and $Z(\SL_n(q))$ has a unique subgroup $Z$ of order $d/r$ and that $\SL_n(q)/Z\cong G$.

If $q=2$, then $d=\gcd(n,q-1)=1$. Therefore, we assume $q>2$ and $d>1$. We know that $\SL_n(q)$ has $q-1$ Weil characters $\tau_{n,q}^i$ (see \cite{TZ2}) with $0\leq i \leq q-2$ of degree $(q^n-1)/(q-1)$ if $1\leq i\leq q-2$ and $(q^n-q)/(q-1) $ if $i=0$.  Using the formula for $\tau_{n,q}^i$ (see \cite{Tiep15}),  the kernel of $\chi=\tau_{n,q}^r$  is a subgroup of $Z(\SL_n(q))$ of order $d/r$, and thus $\chi$ is a faithful irreducible character of $\SL_n(q)/\ker(\chi)\cong G$.

Since $\chi\in\Irr(G)$ is faithful, $\cod(\chi)=|G|/\chi(1)=|H|/\phi(1)$ for some $\phi\in\Irr(H).$ Thus $\chi(1)=r\phi(1)$. Here $\phi$ is nontrivial as $\chi(1)=(q^n-1)/(q-1)>r.$   Assume  $H\cong \LL_4(3)$. Then $G\cong \SL_4(3)\cong 2.\LL_4(3)$ and by using GAP \cite{GAP}, $G$ has a faithful irreducible character of degree $40$ but $40/2=20$ is not a degree of $H$. So, assume $(n,q)\neq (4,3).$
By \cite[Theorem 1.1]{TZ1}, we have $\phi(1)\ge (q^n-q)/(q-1)$ and hence $(q^n-1)/(q-1)=r\phi(1)\ge r(q^n-q)/(q-1)$ which implies that $(q^n-1)\ge r(q^n-q)$ or equivalently $(r-1)(q^n-q)\leq q-1$, which is impossible as $q\ge 3,r\ge 2$ and $n\ge 4$.

\smallskip
(2) Assume  $C_G(N)=N$.
By Lemma \ref{lem:2m}, $N$ is an irreducible $\FF_rH$-module and  $r^{2m}$ divides $ |H|$, hence $2m\leq \log_r(|H|_r)$.

\smallskip
(2a)  Case $r\neq p$. Since $H$ has a faithful irreducible representation of dimension $m$ in cross characteristic $r$, by \cite[Theorem 1.1]{GT99}, we have  $m\ge (q^n-q)/(q-1)-\kappa_n$ if $H$ is not isomorphic to $\LL_6(2)$ nor $\LL_6(3)$; for the exception, we have $m\ge 62-\kappa_6$ and $m\ge 363-\kappa_6$, respectively,
 where $\kappa_n=1$ if $r$ divides $ (q^n-1)/(q-1)$ and $0$ otherwise. 
  
 Assume first that $H\cong \LL_6(2)$.  Here $r>2$,  $|H|=2^{15}\cdot 3^4\cdot 5\cdot 7^2\cdot 31$ and clearly $2m>\log_r(|H|_r)$. Similarly, for  $H\cong \LL_6(3)$, $r\neq 3$, $|H|=2^{11}\cdot 3^{15}\cdot 5\cdot 7\cdot 11^2\cdot 13^2$ and   $2m>\log_r(|H|_r)$. Thus we can assume that  $$m\ge (q^{n}-q)/(q-1)-1=(q^n-2q+1)/(q-1)\ge q^{n-1}-1.$$

(i) Assume $r=2$. Then $q$ is odd and $2\mid (q-1)$. 
Assume $q\equiv 1$ mod $4$. By \cite[Table 2]{Stather} or Lemma \ref{lem:Sylow},  $|H|_2\leq (q-1)_2^{n-1}(n!)_2\leq 2^n(q-1)^n$.  Hence $$\log_2(|H|_2)\leq n+n\log_2(q-1)\leq n+n(q-1)=nq.$$ Since $n\ge 4,$  $q^{n-2}> 2n$ for any $q\ge 3.$ Hence $$2m\ge 2(q^{n-1}-1)\ge 2(q-1)q^{n-2}>4n(q-1)>nq\ge \log_2(|H|_2),$$ which is a contradiction.

Assume $q\equiv 3$ mod $4$ and $n=2m+1.$ Then $$|H|_2\leq |\GL_n(q)|_2/(q-1)_2\leq (m!)_24^m(q+1)_2^m<2^{3m}(2q)^m=2^{4m}q^m.$$ Hence $\log_2(|H|_2)\leq 4m+m\log_2(q)\leq 2n+nq\leq n(q+2)$.  As in the previous case, $2m>4n(q-1)>n(q+2)\ge \log_2(|H|_2)$. The remaining case is similar.

(ii) Assume $r>2$.
Note that $r$ divides $ \prod_{i=1}^n(q^i-1)$.  Let $j=d_q(r)$ be the order of $q$ modulo $r$. Then $j$ is the smallest positive integer such that $r$ divides $q^j-1$. Now $|H|_r=(q^j-1)^k_r(k!)_r$, where $k=\lfloor n/j\rfloor.$
Since $r^{2m}$ divides $ |H|$,  $r^{2m}\leq |H|_r=(q^j-1)^k_r(k!)_r$. 

%Since $m\ge (q^n-q)/(q-1)-\kappa_n$, we have \[ 2(q^{n-1}-1)<2[(q^n-q)/(q-1)-\kappa_n]\leq 2m.\]

 Assume  $j=n$.  Then $n\mid (r-1)$ and so $r\ge n+1$. Now $|H|_r=(q^n-1)_r$. Clearly, $r\nmid (q-1)$ so $r\mid (q^n-1)/(q-1)$. It follows that $\log_r(|H|_r)\leq |H|_r\leq (q^n-1)/(q-1)$. However, $2(q^n-2q+1)/(q-1)> (q^n-1)/(q-1)$ if and only if $q^n> 4q-3$, which is true since $q\ge 2$ and $n\ge 4.$ Therefore, $2m\ge (q^n-2q+1)/(q-1)> (q^n-1)/(q-1)\ge \log_r(|H|_r)$.

 Assume  $n/2<j\leq n-1.$ Then $k=\lfloor n/j\rfloor=1$ and $|H|_r=(q^j-1)_r.$ With the same argument as above, we get $2m\ge 2(q^{n-1}-1)>q^{n-1}-1\ge \log_r(|H|_r)$, a contradiction.

 Assume $1\leq j\leq n/2$.  Then $$\log_r(|H|_r)\leq k\log_r(q^j-1)+\log_r((k!)_r)\leq n/j(q^j-1)+n/j=nq^j/j\leq nq^j\leq nq^{n/2}.$$
 Since $2m\ge 2(q^{n-1}-1)>q^{n-1}$,  it suffices to show that $q^{n-1}>nq^{n/2}$ or equivalently $q^{n-2}>n^2$ for all integers $n\ge 4$ and $q\ge 2$. 
If $q\ge 5$, then $q^{n-2}\ge 5^{n-2}>n^2$ by induction on $n\ge 4.$ So, assume $2\leq q\leq 4.$ We have $q^{n-2}>n^2$ except for the following cases: $q=2 $ and $5\leq n\leq 8$; $q=3$ and $n=4$ or $q=4$ and $n=4.$ For these cases, direct calculation shows that ${2m}>\log_r(|H|_r)$ for every prime divisor $r\neq p$ of $|H|$.

\smallskip
(2b)  Case $r=p.$ Recall that $H\cong \LL_n(q)$, $L\cong \SL_n(q)$ and $L/Z(L)\cong H$, where $q=p^f$. Let $F=\FF_p$, $E=\FF_q$ and  $\mathbf{k}=\overline{\FF}_q$.   Note that $N$ is a faithful irreducible $FH$-module with $\dim_F(N)=m.$

Let $K$ be field extension of $F$ and let $U$ be an irreducible  $KH$-module. Then $U$ is also an irreducible $KL$-module, where the center of $L$ acts trivially on $U.$ Consequently, if $U$ is an absolutely irreducible $KH$-module, then $U$ is an absolutely irreducible $KL$-module. 
Assume $K$ is a finite field extension of $F$ and let $\phi\in \Gal(K/F)$. For each $KH$-module $W$, one can define a $KH$-module $W^\phi$ such that $W$ is irreducible if and only if $W^\phi$ is irreducible and $\dim_K(W)=\dim_K(W^\phi)$ (see \cite[VII.1.13]{HB} or \cite[p. 48]{KL}).

 By the discussion above, $N$ is an irreducible $FL$-module.
Since $L$ is a finite Lie-type group of simply connected type defined over $E=\FF_q$,  by \cite[Table 5.4.C]{KL} the smallest dimension of nontrivial irreducible representations of $L$ over  $\mathbf{k}$ is $n$. Hence it follows from  \cite[Lemma 4.2]{GT05} that  $|N|=p^m \ge q^n$ (we can also get this bound using the existence of primitive prime divisors of $p^{nf}-1$ in $H$, which always exists by Zsigmondy's theorem except for the case $(n,q)=(6,2)$.) It follows that $m\ge nf.$
By Lemma \ref{lem:2m}, we have $2m\leq \log_p(|H|_p)=fn(n-1)/2$.  Hence $2nf\leq fn(n-1)/2$, which implies that $n\ge 5.$

By \cite[Proposition 5.4.4]{KL}, $E$ is a splitting field for $L.$
 Let $\rho:L\rightarrow  \GL_m(F)$ be an $F$-representation of $L$ afforded by $N$.  
 Let $K$ be the centralizer in $\textrm{M}_m(F)$ of $\rho(L)$.  Let $\chi$ be the character of an  irreducible constituent of  $\rho^E$ ($\rho^E$ is just $\rho$ viewed as an $E$-representation of $L$.) By Problems $9.7$ and $9.8$ in \cite{Isaacs},  $K=F(\chi)$ is a finite field extension of $F$ of degree $t$, where $t$ divides $f$ as $F\subseteq K\subseteq E$. Let $W$ be a $KL$-module affording $\chi$.
By Theorems 9.21 and 9.14 in \cite{Isaacs}  (see also \cite[Lemma 2.10.2]{KL} applies to $H\leq GL(N,F)$), $W$ is an absolutely irreducible $KL$-module, $N\otimes_F K\cong \oplus_{i=1}^tW^{\psi^i}$ and $\langle \psi\rangle=\Gal(K/F)$.
Observe that $$m=\dim_F(N)=\dim_K(N\otimes_F K)=t\dim_K(W).$$ Hence $$|W|=|K|^{\dim_K(W)}=|F|^{t\dim_K(W)}=|F|^m=|N|.$$ Next, let $p^e$ be the smallest power of $p$ such that $W$ is absolutely irreducible over $\FF_{p^e}$.
Here $L$ is  simply connected of Lie type defined over $\FF_{p^f}$ and $W$ is an absolutely irreducible $\FF_{p^e}L$-module which is realized over no proper subfield of $\FF_{p^e}$.
By \cite[Proposition 5.4.6(i)]{KL},  $e$ divides $f$ and there exists an irreducible $\mathbf{k}L$-module $M$ such that \[W\otimes \mathbf{k}\cong M\otimes M^{(e)}\otimes \cdots\otimes M^{(f-e)}.\]
In particular, $\dim_{\FF_{p^e}}(W)=\dim_{\mathbf{k}}(W\otimes \mathbf{k})=\dim_{\mathbf{k}} (M)^{f/e}$.  Hence $$p^m=|N|= |W|= p^{e\dim_{\mathbf{k}}(M)^{f/e}}.$$ Thus $m= e\dim_{\mathbf{k}}(M)^{f/e}$. Let $s=f/e$.
Following \cite{GT05}, we call the Frobenius twists of restricted modules of groups of Lie type `restricted'.   Note that $\dim_{\mathbf{k}}(M)\ge n$.

Assume $s\ge 3$. Here $m\ge en^s.$ We  claim that $en^s\ge n^2f$. This is equivalent to $n^{s-2}\ge s$, where $n\ge 5$ and $s\ge 3$. This can be proved easily by induction on $s\ge 3.$ Hence $2fn^2\leq 2m\leq fn(n-1)/2$, which is equivalent to $4n\leq n-1$, which is impossible.

Assume $s=2$. Assume $M$ is not restricted. Then $\dim_{\mathbf{k}}(M)\ge n^2$ and hence $m\ge 2n^2f$. In this case, $2m\ge 4n^2f>fn(n-1)/2$.  Assume $M$ is restricted. Then $W\otimes {\mathbf{k}}\cong M\otimes M^{(f/2)}$ and the smallest field over which $W$ can be realized is $\FF_{p^{f/2}}.$  Again, we get $m\ge en^2=n^2f/2$, whence $2m\ge n^2f>n(n-1)f/2= \log_p(|H|_p)$, a contradiction.

Assume $s=1$. Here $e=f.$ If $M$ is a tensor product of at least two nontrivial restricted irreducible $\mathbf{k}L$-modules, then $m\ge n^2f$ and we get a contradiction as in the previous cases. Assume $M$ is a restricted irreducible $\mathbf{k}L$-module.
Then  $m= f\dim_{\mathbf{k}}(M)$.  Assume first that $\dim_{\mathbf{k}}(M)>n$. By  \cite{Lubeck} or \cite[Table 5.4.A]{KL}, $\dim_{\mathbf{k}}(M)\ge n(n-1)/2$ and thus $2m\ge fn(n-1)>fn(n-1)/2= \log_p(|H|_p)$, a contradiction.
We are left with the case when $W\otimes \mathbf{k} \cong M$ is a restricted $\mathbf{k} L$-module with  $\dim_{\mathbf{k}}(W)=n$ and hence $M$ is quasi-equivalent to the natural module for $L$. It follows that $|N|=|W|=q^n$. 

\smallskip
Recall that if $X$ is a finite group and $A$ is an $X$-module, let $H^2(X,A)$ be the second cohomology group of $X$ with coefficients in $A$. It is well-known that $H^2(X,A)$ is nonzero if and only if there is a nonsplit extension of $A$ by $X.$ 
From \cite[Table I]{Bell}, we know that since $n\ge 5,$ $H^2(\SL_n(q),V)$ is trivial, where $V$ is the natural module for $\SL_n(q)$, except for the case $H\cong \SL_5(2)$. 

Assume that $H\cong \SL_5(2)$ and $N$ is not quasi-equivalent to the natural module for $H$ so that $G$ is isomorphic to the  nonsplit extension $2^5.\GL_5(2)$ which is called the Dempwolff group and appears as a maximal subgroup of the Thompson sporadic simple group Th.
By GAP \cite{GAP}, we can check that this group has a faithful irreducible character $\chi$ of degree $248=2^3\cdot 31.$ Clearly $\chi(1)/2^5$ is not a character degree of $H$ and thus this case cannot occur.
By using GAP \cite{GAP}, we can also assume that $H\not\in\{\LL_5(2),\LL_6(2)\}$. 

Assume $f=1$. Then $N$ is quasi-equivalent to the natural module for $\SL_n(p)$ and hence $\SL_n(p)$ has exactly one nontrivial orbit on $N$ and since $N\cong \Irr(N)$, $\SL_n(p)$ has exactly one nontrivial orbit on $\Irr(N)$ of size $p^n-1$. Since $H$ is perfect and embeds into $\GL_n(p)$, we see that $H$ is isomorphic to a subgroup of $\SL_n(p)$. Since $\SL_n(p)$ is quasisimple as $n\ge 5,$ and contains a subgroup isomorphic to $\LL_n(p)$, it forces $H\cong \SL_n(p)$. In particular, $\gcd(n,p-1)=1$.

Therefore, $G= N:H_0$ with $H_0\cong H\cong\SL_n(p)$. Furthermore, since $N\cong \Irr(N)$, $G$ has exactly one nontrivial orbit $\mathcal{O}$ on $\Irr(N)$ of size $p^n-1.$ Let $\theta\in\mathcal{O}$ and  $T=I_G(\theta)\cong N:T_0,$ where $T_0$ is  the stabilizer of a vector in the natural module $N$. By \cite[Proposition 4.1.17]{KL}, we deduce that $T_0\cong C:A$, where $C$ is an elementary abelian $p$-group of order $p^{n-1}$ and $A\cong \SL_{n-1}(p)$. Since $\theta$ is $T$-invariant and $T$ splits over $N$, by \cite[Lemma 22.2]{Hupp98}, $\theta$ extends to $\theta_0\in\Irr(T)$ and so since $T_0/C\cong \SL_{n-1}(p)$, $T_0$ has an irreducible character of degree $(p^{n-1}-p)/(p-1)$ (the degree of a unipotent character of $\SL_{n-1}(p)$) and hence by Gallagher's Theorem \cite[Corollary 6.17]{Isaacs}, $T$ has an irreducible character $\phi$ lying over $\theta$ of degree $(p^{n-1}-p)/(p-1)$. It follows that $\chi=\phi^G\in \Irr(G|\theta)$ is faithful and hence by Lemma \ref{lem:2m},  $|N|=p^n$ divides $\chi(1)=|G:T|\phi(1)=(p^n-1)(p^{n-1}-p)/(p-1)$, which is impossible as $n\ge 5.$ (We can also take $\theta_0\in\Irr(I_G(\theta))$ that is an extension of $\theta$ to $I_G(\theta)$ and thus $\theta_0^G\in\Irr(G|\theta)$ with degree $p^n-1$ and get a contradiction as well.)

Now assume that $f>1$. Note that  $(nf,p)\neq (6,2)$ since $n\ge 5$ and $f\ge 2$.  Thus a primitive prime divisor of $p^{nf}-1$ exists and divides $|H|$. Since $H\cong \LL_n(p^f)$ is a subgroup of $\SL_{m}(p)=\SL_{nf}(p)$, by \cite[Section 3]{Ber} or \cite{GPPS}, $H$ must be isomorphic to a $\CCC_3$-subgroup of $\SL_m(p)$ and thus $H\cong \SL_k(p^e)$ with $ke=m=nf$. It follows that $k=n$ and $e=f$ is the only possibility and thus $H\cong \SL_n(q)$ which also forces $\gcd(n,q-1)=1$. Furthermore, $H$ acts transitively on the nonzero vectors of $N$ as a natural module for $\SL_m(p)$. (See \cite[Appendix 1]{Lie87}).
From the construction of $\CCC_3$-subgroups of $\SL_m(p)$,  $N$ can be considered  an $\FF_q$-vector space of dimension $n$ and it is in fact the natural module for $\SL_n(q)$. (We start with the natural module $N$ for $\SL_n(q)$ of dimension $n$ over $\FF_{p^f}$ and then consider $N$ as an $\FF_p$-module of dimension $m=nf$.) We can then apply the same argument as in the case $f=1$ to obtain a contradiction.
The proof is now complete.
\end{proof}

Theorem \ref{th:main1} now follows from Theorems \ref{th:exceptional} and \ref{th:Ln}.

\section{Symplectic groups}\label{sec6}

Let $L=\Sp_{2n}(q)$, where $n\ge 2$ and $q=p^f$. If $n=2$, we assume $q>2.$ Then $H=L/Z(L)\cong \SSS_{2n}(q)$ with $d=|Z(L)|=\gcd(2,q-1)$. 
Let $D$ be the dimension of the natural module for $L$, so $D=2n$, and it is also the smallest dimension of a nontrivial absolutely irreducible $\mathbf{k}L$-module, where $\mathbf{k}=\overline{\FF}_q$. Let $F=\FF_p$ and $E=\FF_q.$ We have 
\[|\SSS_{2n}(q)|=\frac{1}{d}q^{n^2}\prod_{i=1}^n(q^{2i}-1).\]
Moreover, for $(n,q)\neq (3,2)$, we have $|M(H)|=d$ while $M(\SSS_6(2))\cong \ZZ_2$ (see \cite[Theorem 5.1.4]{KL}).
In this section, we study the structure of minimal counterexamples to Conjecture \ref{conj:codegrees} for the finite simple symplectic groups and as a consequence, confirming the conjecture for $n=2$ and $n=3.$ 
Unfortunately, we cannot complete the verification of the conjecture for these groups since the second cohomology group $H^2(\Sp_{2n}(q),V)$ may not be trivial, where $V$ is the natural module for $\Sp_{2n}(q)$. In particular, the extension of $H$ by $N$ is not a split extension and thus our approach does not work.

\begin{thm}\label{th:PSp}
 Let $H=\SSS_{2n}(q)$, where $n\ge 2,q=p^f$ and $(n,q)\neq (2,2).$ Let $G$ be a counterexample to Conjecture \ref{conj:codegrees} with minimal order. Let $L=\Sp_{2n}(q)$. Then  $n\ge 4$ and the following hold.
 \begin{enumerate}
 \item $G$ is perfect and has a unique minimal normal subgroup $N$ which is an elementary abelian $p$-group.
 \item All the irreducible $\mathbf{k}L$-submodules of $N\otimes \mathbf{k}$ is quasi-equivalent to the natural module for $L$.
 \item $|N|=q^{2n}$.
 
 \end{enumerate}
\end{thm}

\begin{proof}
Let $G$ be a counterexample to the theorem with minimal order.  Then $\cod(G)=\cod(H)$ but $G\not\cong H$.  By  Lemma \ref{lem:minimal-normal}, $G$ is perfect and has a unique minimal normal elementary abelian  $r$-subgroup $N$ of order $r^m$ with $G/N\cong H$. Now either $G$ is  quasisimple  with center $N$ of order $r$ or $C_G(N)=N$.

\smallskip

(1) Assume that $G$ is quasisimple with center $N\cong \ZZ_r.$  
Assume first that $H=\SSS_6(2)$. Then $G\cong 2.\SSS_6(2)$ and by using GAP \cite{GAP}, $G$ has a faithful irreducible character of degree $8$ but $8/2=4$ is not a character degree of $H$, so $G$ cannot be $2.\SSS_6(2)$ by Lemma \ref{lem:quasisimple}.
Now assume that $(n,q)\neq (3,2)$. Then $G\cong \Sp_{2n}(q)$, $r=2$ and $q$ is odd. By \cite[Lemma 2.6]{TZ2}, $\Sp_{2n}(q)$ always has a faithful irreducible character of degree $(q^n-\alpha)/2$, where $4\mid (q^n-\alpha)$ and $\alpha\in \{\pm 1\}$. Now $(q^n-\alpha)/4$ is not a degree of $H$ since the minimal nontrivial degree of $H$ is at least $(q^n-1)/2$ by \cite[Table 5.3.A]{KL}.

(2) Assume $C_G(N)=N$. 
By Lemma \ref{lem:2m}, $N$ is an irreducible $\FF_rH$-module and $r^{2m}$ divides $|H|$. Hence $2m\leq \log_r(|H|_r)$.

\smallskip
(2a) Case $r\neq p$.

Assume that $H\cong \SSS_6(2)$. Then $m\ge 7$ by \cite[Table 5.3.A]{KL}. We have $|\SSS_6(2)|=2^9\cdot 3^4\cdot 5\cdot 7$. We can see that $r^{2m}$ cannot divide $|H|$ for any odd prime $r$. Therefore, we assume $H\not\cong \SSS_6(2).$
By  \cite[Table 5.3A]{KL}),  $m\ge (q^n-1)/2$ if $q$ is odd and $m\ge q^{n-1}(q^{n-1}-1)(q-1)/2$ if $n$ is even. In both cases,  $2m\ge q^n-1.$

(i) Assume  $r=2$. So $q$ is odd and $|H|_2=(q^2-1)_2^n(n!)_2/2< (q^2-1)^n2^{n-1}$ and so $\log_2(|H|_2)< n-1+n(q^2-1)= nq^2-1$. If $n\ge 3$, then  $2m\ge q^{n}-1\ge nq^2-1> \log_2(|H|_2)$  for any $q\ge 3$ as $q^{n-2}\ge 3^{n-2}\ge n$. Assume $n=2.$  Then $|H|_2\leq (q^2-1)_2^2$ and so $\log_2(|H|_2)< 2\log_2(q^2-1)$. Thus we need to show that $2m\ge q^2-1>2\log_2(q^2-1)$ which is always true since $q^2-1\geq 8.$

(ii) Assume $r\ge 3$ is odd. Then $r$ divides $\prod_{i=1}^n(q^{2i}-1)$.
Let $j=d_q(r)$ be the order of $q$ modulo $r$. Then $1\leq j\leq 2n$ and $r\ge j+1$.

(ii)(1) Assume $j$ is odd. By \cite[Table 1]{Stather}, $|H|_r=|\GL_n(q)|_r\leq (q^j-1)_r^kr^{k/(r-1)}$ and hence $\log_r(|H|_r)\leq k/(r-1)+k\log_r(q^j-1)$. 

Assume $j=n$ is odd. Then $r\ge n+1$ and so $|H|_r=(q^n-1)_r$. Since $r\nmid (q-1)$, we have $\log_r(|H|_r)\leq (q^{n}-1)/(q-1)$. If $q>2$, then $2m\ge q^n-1>(q^{n}-1)/(q-1)\ge \log_r(|H|_r)$. If $q=2$, then $2m\ge 2^{n-1}(2^{n-1}-1)>2^n-1\ge \log_r(|H|_r)$.

Assume $n-1\ge j> n/2$. Then $k=\lfloor n/j\rfloor=1$ and thus $|H|_r\leq (q^j-1)_r$. Hence $\log_r(|H|_r)\leq q^{j}-1\leq q^{n-1}-1$. Now $2m\ge q^n-1>q^{n-1}-1\ge \log_r(|H|_r)$.

Assume $1\leq j\leq n/2$. Then $|H|_r\leq (q^j-1)^k_r(k!)_r\leq (q^j-1)_r^kr^{k/(r-1)}$ and thus $\log_r(|H|_r)< k(q^j-1)+k=kq^j\leq nq^{n/2}.$

Assume that $q\ge 3$. Then  $2m\ge q^n-1>q^n/2$. Observe that $q^n/2\ge nq^{n/2}$ is equivalent to $q^{n}\ge 4n^2$ which is always true for any $q\ge 3$ and $n\ge 4.$ Therefore,  $2m>q^n/2\ge nq^{n/2}\ge \log_r(|H|_r)$.
Assume $n=2$. Then $j=1$ and $\log_r(|H|_r)=2\log_r((q-1)_r)\leq 2(q-1)$ and so $2m\ge q^2-1>2(q-1)\ge \log_r(|H|_r)$. Similarly, if $n=3$, then $j=1$ and $\log_r(|H|_r)\leq 1+3\log_r((q-1)_r)\leq 1+3(q-1)<3q$ and so $2m\ge q^3-1>3q\ge \log_r(|H|_r)$.

Assume $q=2$. Then $n\ge 4$ and $2m\ge 2^{n-1}(2^{n-1}-1)/2= 2^{n-2}(2^{n-1}-1)$. Now $2^{n-2}(2^{n-1}-1)>n2^{n/2}$ if and only 
$2^{(n-4)/2}(2^{n-1}-1)>n$. Since $n\ge 4,$  $2^{n-1}>n$. Therefore, $2m\ge 2^{n-2}(2^{n-1}-1)>n2^{n/2}\ge \log_r(|H|_r),$ a contradiction.

(ii)(2) Assume $j$ is even.  So $(q^j-1)_r=(q^{j_0}+1)_r$. By \cite[Table 1]{Stather}, $|H|_r=|\GL_{2n}(q)|_r\leq (q^j-1)_r^k(k!)_r\leq (q^j-1)^k_rr^k$, where $k=\lfloor 2n/j\rfloor$ and hence $\log_r(|H|_r)\leq k+k\log_r(q^j-1)_r$. 

 Assume $j=2n.$ Then $r\mid (q^{n}+1)$ and $r\ge 2n+1$ so $|H|_r=(q^{2n}-1)_r=(q^n+1)_r$. Assume $q$ is odd. Then $(q^n+1)_r=((q^n+1)/2)_r$ since $r$ is odd. Thus $2m\ge q^n-1>(q^n+1)/2\ge \log_r(|H|_r)$. Assume $q$ is even. If $q=2$, then $2m\ge 2^{n-1}(2^{n-1}-1)/2=2^{n-2}(2^{n-1}-1)>(2^n+1)\ge \log_r(|H|_r)$  since $n\ge 4.$ Assume $q\ge 4$ is even. Then $2m\ge q^{n-1}(q^{n-1}-1)(q-1)/2>q^{n-1}(q^{n-1}-1)>q^n+1\ge \log_r(|H|_r)$.

 Assume $2n> j>n$. Then $k=1$. We have $|H|_r=(q^j-1)_r=(q^{j_0}+1)_r$ and so $\log_r(|H|_r)=\log_r(q^{j_0}+1)\leq q^{n-1}+1$.
Now $2m\ge q^n-1>q^{n-1}-1\ge \log_r(|H|_r).$

 Assume $n\ge j\ge 2.$ Then $k\leq n$ and  $|H|_r\leq  (q^j-1)_r^k(k!)_r\leq (q^{j_0}+1)^k_rr^k $ and hence $\log_r(|H|_r)\leq k+k(q^{j_0}+1)\leq k(q^{n/2}+2)\leq n(q^{n/2}+2)\leq 2nq^{n/2}$ as $n\ge 2$ and $q\ge 2$.

Assume first that $q$ is even. So $2m\ge q^{n-1}(q^{n-1}-1)> q^{2n-2}/2$. Now observe that $q^n\ge 4n$ for  any $n\ge 4$ and $q\ge 2$. Hence $q^{2n-2}/2=q^nq^{n-2}/2\ge q^n/2\ge 2nq^{n-2}\ge  2nq^{n/2}$ and thus  $2m>\log_r(|H|_r)$. 
Assume that $n=2$ or $3$. Then $q\ge 4$ and clearly $q^n\ge 4^n\ge 4n$ and again $2m>\log_r(|H|_r)$ as above.

Assume $q$ is odd. 
Assume  $n=2$. Then $j=2$ and $2m\ge q^2-1$. Now  $\log_r(|H|_r)= 2\log_r((q^2-1)_r)\leq 2(q+1)$. However, we see that $2m\ge q^2-1>2(q+1)\ge \log_r(|H|_r)$. Assume next that $n=3$. Again, $j=2$, $2m\ge q^3-1$ and $\log_r(|H|_r)=1+3\log_r((q^2-1)_r)\leq 1+3(q+1)<q^3\leq 2m.$ Assume that $n=4$. Then $j=2$ or $4$ and $2m\ge q^4-1$. If $j=2$, then $\log_r(|H|_r)\leq 1+4(q+1)<q^4-1\leq 2m.$ If $j=4$, then $\log_r(|H|_r)\leq 2(q^2+1)<q^4-1\leq 2m.$
Assume $n\ge 5.$ By induction on $n\ge 5$,  $q^{n}> 4n^2$ for $q\ge 3$ and thus $q^{n/2}> 2n$. Finally $q^n-1>2nq^{n/2}$ is equivalent to $q^{n/2}(q^{n/2}-2n)>1$ which is true by the previous claim. Thus $2m\ge q^n-1>2nq^{n/2}\ge \log_r(|H|_r)$.

\smallskip
(2b) Case $r=p$.  Let $L\cong \Sp_{2n}(q)$ with center $Z(L)$ of order $\gcd(2,q-1)$ and $L/Z(L)\cong H.$ Again, let $E=\FF_q$,  $F=\FF_p$ and $\mathbf{k}=\overline{\FF}_q$, where $q=p^f$. Recall that $N$ is an irreducible $FH$-module and hence it is an irreducible $FL$-module.

By \cite[Table 5.4.C]{KL},  the smallest dimension $D$ of nontrivial irreducible representations of $L$ over $\mathbf{k}$ is $2n$.  By \cite[Lemma 4.2]{GT05}, $|N|\ge q^D$ and hence $m\ge 2nf$. As $2m\le \log_p(|H|_p)$, we get $2m\leq n^2f$
and thus $n^2f\ge 4nf$, so $n\ge 4.$

Since $D=2n\ge 8$, we can apply \cite[Lemma 4.2]{GT05} directly and hence cases (i), (ii) or $(v)$ holds. If (v) holds, then $|N|\ge q^{2D^2}$ and hence $m\ge 2fD^2\ge 8n^2f$. If $(ii) $ holds, then $f$ is even and $m\ge (f/2)(2n)^2=2n^2f$. Finally, if  (i) holds and $N\otimes \mathbf{k}$ is a tensor product of at least two restricted irreducible $\mathbf{k}L$-modules, then $m\ge 4n^2f$. For these cases, we see that $m\ge 2fn^2$ and so $2m\ge 4n^2f>n^2f$ and thus $p^{2m}$ cannot divide $|H|$. Therefore, we are left with the case when $N\otimes \mathbf{k}$ is a restricted $\mathbf{k}L$-module and $f=e$. Assume $n\ge 7$. By \cite{Lubeck} or \cite[Table 5.4.A]{KL}, the dimension of a restricted $\mathbf{k}L$-module is either $2n$ or at least $2n^2-n-2$.  If the latter case holds, then $2m\ge 2(2n^2-n-2)f\ge 2n^2f>n^2f$ and thus $p^{2m}\nmid |H|$. Therefore, $m=2nf$ and $N\otimes \mathbf{k}$ is quasi-equivalent to the natural module for $L$. Hence $m=2nf$ and $|N|=q^{2n}$.

Assume $4\leq n\leq 6.$ If $q$ is odd, then either $m=2nf$ or $m\ge (2n^2-n-2)f$ and as above, the former case holds. 
Assume $q$ is even. Then $L$ has a spin module of dimension $2^n$ (see \cite[Table 5.4.A]{KL}).  If $m\ge 2^nf$, then $2m\ge 2^{n+1}f$ and we can check that $2^{n+1}>n^2$ for these values of $n$ and hence $2m>n^2f$. Thus $m=2nf$ as wanted.
\end{proof}

\section*{Acknowledgment}  The author is grateful to  the reviewers for suggestions and corrections that improve the exposition of the paper. 

\section*{Declarations}  

\textbf{Conflict of Interest} The author has no conflict of interest to declare that are relevant to this article.

\textbf{Funding} The author declares that no funds, grants, or other support were received during the preparation of this manuscript.

 \textbf{Data Availability}  Data sharing not applicable to this article as no datasets were generated or analyzed during the current study.

\end{document}